\documentclass{article}

\usepackage{amsfonts, amsmath, amssymb, graphicx, pstool}

\usepackage[normalem]{ulem}   

\usepackage[hyperref,amsmath,thmmarks]{ntheorem}
\usepackage{aliascnt}
\newtheorem{lemma}{Lemma}[section]

\newaliascnt{proposition}{lemma}
\newtheorem{proposition}[proposition]{Proposition}
\aliascntresetthe{proposition}

\newaliascnt{corollary}{lemma}

\aliascntresetthe{corollary}

\newaliascnt{theorem}{lemma}
\newtheorem{theorem}[theorem]{Theorem}
\aliascntresetthe{theorem}

\newtheorem{example}[theorem]{Example}

\newtheorem{proof}{Proof}

\def\fin{\ifmmode{\Large$\diamond$}\else{\unskip\nobreak\hfil
    \penalty50\hskip1em\null\nobreak\hfil{\Large$\diamond$}
    \parfillskip=0pt\finalhyphendemerits=0\endgraf}\fi}

\mathchardef\Gamma="0100
\mathchardef\Xi="0104
\def\be#1#2\ee{\begin{equation}\label{eq:#1}#2\end{equation}}
\def\req#1{{\rm(\ref{eq:#1})}}
\def\bdm  {\begin{displaymath}}
\def\edm  {\end{displaymath}}
\def\bdmal{\begin{displaymath}\begin{aligned}}
\def\edmal{\end{aligned}\end{displaymath}}

\newcommand{\N}{{\mathord{\mathbb N}}}
\newcommand{\R}{{\mathord{\mathbb R}}}

\newcommand{\X}{{\cal X}}
\newcommand{\Y}{{\cal Y}}
\newcommand{\E}{{\cal E}}
\newcommand{\xlsp}{x_n^{I}}
\newcommand{\xdls}{x_n^{I\!I}}
\newcommand{\xbayes}{x_n^{I\!I\!I}}
\newcommand{\mean}{x^\circ}
\newcommand{\Mean}{X^\circ}
\newcommand{\ymean}{y^\circ}

\newcommand{\xx}{{{\boldsymbol\xi}}}

\newcommand{\norm}[1]{\|#1\|}

\newcommand{\scalp}[1]{\langle\,#1\,\rangle}

\newcommand{\hull}{\operatorname{span}}
\newcommand{\rge}[1]{{\cal R}(#1)}

\newcommand{\yd}{y}
\newcommand{\xdagger}{x^\dagger}

\def\req#1{{\rm(\ref{eq:#1})}}

\makeatletter
\newcommand{\dupdots}{\mathinner{\mkern1mu\raise\p@
    \vbox{\kern7\p@\hbox{.}}\mkern2mu
    \raise4\p@\hbox{.}\mkern2mu\raise7\p@\hbox{.}\mkern1mu}}
\makeatother
\makeatletter
\newcount\c@MaxMatrixCols \c@MaxMatrixCols=10
\newenvironment{cmatrixc}{%
  \hskip -\arraycolsep\array{*\c@MaxMatrixCols c}%
}{%
  \endarray \hskip -\arraycolsep
}
\makeatother
\newenvironment{cmatrix}{\left[\cmatrixc}{\endmatrix\right]}

\newenvironment{keywords}{
    \vspace{1em}\noindent\textbf{Keywords:}\quad
}{
    \vspace{1em}
}

\title{Addendum on data driven regularization by projection}
\author{Martin Hanke\thanks{Institute of Mathematics, Johannes
    Gutenberg-Universit\"at Mainz, 55099 Mainz, Germany
    ({\tt hanke@math.uni-mainz.de}).} 
    \and Otmar Scherzer\thanks{Faculty of Mathematics, University of Vienna, Oskar-Morgenstern-Platz 1, A-1090 Vienna, Austria \& Johann Radon Institute for Computational and Applied Mathematics (RICAM), Altenbergerstraße 69, A-4040 Linz, Austria \& Christian Doppler Laboratory for Mathematical Modeling and Simulation of Next Generations of Ultrasound Devices (MaMSi), Oskar-Morgenstern-Platz 1, A-1090 Vienna, Austria ({\tt otmar.scherzer@univie.ac.at}).} }

\begin{document}

\sloppy
\maketitle

\begin{abstract}We study the stability of \emph{regularization by projection} 
for solving linear inverse problems if the forward operator is given indirectly but specified via some input-output training pairs. 
We extend the approach in \cite{AKS20} to data pairs, which are noisy and, possibly,
linearly dependent.
\end{abstract}



Last modified: \today

\pagestyle{myheadings}
\thispagestyle{plain}
\markboth{M. HANKE AND O. SCHERZER}{DATA DRIVEN REGULARIZATION}

\section{Introduction}
Data driven regularization by projection attempts to reconstruct a solution of a linear operator equation 
	\be{op}
	   Kx=y
	\ee
without making explicit use of the operator $K$ but only of some of its input-output relations $(x_k)_k \subseteq \X$ and 
$(y_k)_k \subseteq \Y$.

In \cite{AKS20} we have considered two different strategies how the solution of \req{op}
can be computed in a data driven manner:
	\begin{enumerate}
		\item Method I - least squares projection (called \emph{regularization
        by projection} in \cite{AKS20});
		\item Method II - dual least squares method, originally suggested by Natterer \cite{Natt77}, which, however, requires additional measurements and therefore limits the practical applicability in a data driven setting.
\end{enumerate}
In addition to the methods and results from \cite{AKS20} we consider here the following extensions:
\begin{enumerate}
	\item We introduce a third method (Method~III), which employs a Bayesian approach, and -- in a way -- combines the merits of Methods~I and II.
	\item While in \cite{AKS20} the explicit assumption has been made that the training data $(y_k)_k$ satisfy
	\be{trno}
	Kx_k = y_k\,,
	\ee
we omit this simplifying assumption in this note. But we emphasize that if the training data are exact and if $K$ is injective, then Method~III reduces to Method~I.
\item We also provide a preliminary error analysis for Methods~I and III.
\item In the appendix we give some simple examples concerning the rank of 
covariance matrices. 
\end{enumerate}
%

\section{Setting}
\label{Sec:Setting}

These notes are concerned with the following setting from \cite{AKS20}.

Let $K:\X\to\Y$ be a bounded linear operator between the Hilbert spaces $\X$ 
and $\Y$, which has non-closed range $\rge{K}$. 
Assume further, that the only available information about $K$ 
is given in terms of two sequences $(x_k)_k\subset\X$ -- called \emph{images} in \cite{AKS20} -- and data
$(y_k)_k\subset\Y$ with
\be{noisy}
   \norm{y_k - Kx_k}_\Y\,\leq\, \delta\,, \qquad k=1,2,3,\dots\,,
\ee
for some noise level $\delta>0$. Compared to \cite{AKS20} the data can be noisy, and neither the images nor the data need to be linearly independent.

Let
\bdm
   \X_n \,=\, \hull\{x_1,\dots,x_n\}\,, \qquad n\in\N\,,
\edm
denote the increasing subspaces of $\X$ associated with the given images $x_k$. Further, let $P_{\X_n}$ be the orthogonal projector of 
$\X$ onto $\X_n$. In order to work with the given training pairs we first need an
approximate representation $K_n$ of $KP_{\X_n}$.

To this end define the ''matrix''
\bdm
   X_n \,=\, [x_1,\dots,x_n]\,,
\edm
i.e., the linear map from $\R^n$ to $\X$ with
\be{Xn}
   X_n\xx \,=\, \sum_{k=1}^n \xi_k x_k\,\in\X_n\,, \qquad
   \xx\,=\,\bigl[\xi_k\bigr]_{k=1}^n\in\R^n\,,
\ee
and its adjoint
\bdm
   X_n^*:x\,\mapsto\,\bigl[\,\scalp{x,x_k}_\X\,\bigr]_{k=1}^n\,.
\edm
Further, let
\be{SVD}
   X_n \,=\, U_n \Lambda \Xi^*
\ee
be the singular value decomposition of $X_n$, 
where $\Lambda$ is a diagonal $p\times p$ matrix with the positive singular
values on its diagonal ($p\leq n$ is the dimension of $\X_n$); furthermore,
$\Xi\in\R^{n\times p}$ has orthonormal columns, while $U_n$ is a linear map from $\R^p$
to $\X$ in analogy to $X_n$, and the columns $u_k$ of
\bdm
   U_n \,=\, [u_1,\dots,u_p]
\edm
provide an orthonormal basis of $\X_n$. 
It follows that every $x\in\X_n$ can be written as $U_n\xx$ for some 
$\xx\in\R^p$ with
\bdm
   \norm{x}_\X \,=\, \norm{\xx}_2\,.
\edm

To compute approximations $z_k\approx Ku_k$, $k=1,\dots,p$,
collect the given data from $\Y$ and the unknown $z_k$ in two ``matrices'' 
\be{YandZ}
   Y_n \,=\, [ y_1,\dots,y_n] \qquad \text{and} \qquad 
   Z_n \,=\, [ z_1,\dots,z_p]\,,
\ee
and denote by $e_k$ the $i$th Cartesian basis vector in $\R^n$ for 
$k=1,\dots,n$. Then, the singular value decomposition~\req{SVD} yields
\bdm
   y_k - Kx_k \,=\, (Y_n - KX_n)e_k \,=\, (Y_n-KU_n\Lambda\Xi^*)e_k\,,
\edm
and hence we obtain
\bdm
   \sum_{k=1}^n \norm{y_k-Kx_k}_\Y^2 \,=\, \norm{Y_n - KU_n\Lambda\Xi^*}_{HS}^2\,,
\edm
where $\norm{\,\cdot\,}_{HS}$ is the Hilbert-Schmidt norm 
of the corresponding operator. In view of \req{noisy} and \req{YandZ} it is 
now natural to define $Z_n$ as the minimizer of the functional
\be{Phi}
   \Phi(Z) \,=\, \norm{Y_n \,-\, Z\Lambda\Xi^*}_{HS}
\ee
over $Z\in\Y^p$. Let $\Xi'\in\R^{n\times(n-p)}$ be such that the columns of
$[\Xi\ \Xi']$ form a full orthonormal basis of $\R^n$. Then
\be{Phi2}
   \Phi^2(Z)
   \,=\,
   \norm{(Y_n - Z\Lambda\Xi^*)[\Xi\ \Xi']}_{HS}^2 \,=\,
   \norm{Y_n\Xi-Z\Lambda}_{HS}^2 \,+\, \norm{Y_n\Xi'}_{HS}^2 \,.
\ee
From this we conclude that the unique minimizer of $\Phi$ is given by
\be{Zn}
   Z \,=\, Z_n \,:=\, Y_n\Xi\Lambda^{-1}\,,
\ee
and therefore our best approximation of $K$ on $\X_n$ is given by
\be{Kn}
   K_n \,:=\, Z_nU_n^* \,=\, Y_n\Xi\Lambda^{-1}U_n^* 
\ee
as approximation of $KP_{\X_n}$, once the first $n$ training pairs $(x_k,y_k)$,
$k=1,\dots,n$, have been seen.

In view of \req{SVD} we can rewrite \req{Kn} as
\be{Kn2}
   K_n \,=\, Y_nX_n^\dagger \,=\, Y_nX_n^*(X_nX_n^*)^\dagger
\ee
in terms of the Moore-Penrose generalized inverse $X_n^\dagger$ of $X_n$, 
and we mention that this is a well-known estimator of $K$ in statistics;
cf., e.g., Crambes and Mas~\cite{CrMa13} and 
H\"ormann and Kidzi\'{n}ski~\cite{HoKi15}, and the references therein.
We further note that for exact data, i.e., when $\delta=0$ in \req{noisy}, 
then we have
\be{KnK}
   K_n \,=\, KX_n\Xi\Lambda^{-1}U_n^* \,=\, KU_n\Lambda\Xi^*\Xi\Lambda^{-1}U_n^*
   \,=\, KP_{\X_n}\,.
\ee
In the general case, however, it follows from \req{Kn}, the triangle inequality, and \req{Phi2} that
\begin{align*}
   \norm{(K_n-K)U_n\Lambda}_{HS}
   &\,=\, \norm{(Z_n-KU_n)\Lambda}_{HS} \\[1ex]
&\,\leq\, \norm{Z_n\Lambda-Y_n\Xi}_{HS} \,+\, \norm{Y_n\Xi - KU_n\Lambda}_{HS}
    \\[1ex]
   &\,\leq\, \Phi(Z_n) \,+\, \norm{Y_n\Xi - KU_n\Lambda}_{HS}\;.
\end{align*}
Again, applying \req{Phi2} to the second term in the last line we get 
$\norm{Y_n\Xi - KU_n\Lambda}_{HS} \leq \Phi(KU_n)$, and since $Z_n$ is a minimizer 
of $\Phi$ it follows from \req{SVD} and \req{noisy} that
\be{opnorm}
   \norm{(K_n-K)U_n\Lambda}_{HS}
    \,\leq\, 2\,\Phi(KU_n) \,=\, 2\,\norm{Y_n-KX_n}_{HS}
   \,\leq\, 2\sqrt{n}\,\delta\,.
\ee
Accordingly, since the singular values in $\Lambda$ 
can be much smaller than the noise level when some images $x_k$ are 
almost linearly dependent, the approximation $K_n\approx K$ on $\X_n$ 
can be arbitrarily bad in the noisy case. For example, if $u_k$ is one of the singular functions of $X_n$ and
$\sigma_k$ the associated singular value, cf.~\req{SVD}, then \req{opnorm} gives
\be{sigmak}
\norm{(K_n-K)u_k}
\,\leq\, 2\frac{\sqrt{n}\,\delta}{\sigma_k} \,.
\ee
For a remedy out of this inherent difficulty one can resort, e.g., to using
a truncated singular value decomposition of $X_n$ in \req{Kn2},
cf.~\cite{CrMa13,HoKi15}.

Assume next that noisy data 
\be{eta}
   \yd \,=\ K\xdagger \,+\, \eta
\ee
be given for some $\xdagger\in\X$ and some noise term $\eta$. 
Even if $K$ were known exactly,
the problem of determining stable approximations of $\xdagger$ 
is known to be ill-posed, and we therefore want to employ some kind of
regularization by projection.
In \cite{AKS20} two classical approaches of this sort, investigated in detail
in \cite[Section~3.3]{EHN96}, have been discussed for the above setting.

Method I, known as least-squares projection, defines
\be{lsp}
   \xlsp \,=\, (KP_{\X_n})^\dagger \yd\,, 
\ee
respectively its approximation 
\bdm
   \xlsp \,=\, K_n^\dagger \yd 
\edm
(which is exact according to \req{KnK} in the case of noise-free data, provided that $\xdagger\in \X_n$)
in terms of $K_n$ of \req{Kn}
, where the Moore-Penrose generalized inverse $K_n^\dagger$ of $K_n$ is 
typically computed via the singular value decomposition of $K_n$.

The shortcoming of Method I is the fact that it is known 
(Seidman~\cite{Seid80}) that even for exact data pairs, 
i.e., for $\delta=0$ and $\eta=0$, 
the least-squares projections $\xlsp$ may fail to converge to $\xdagger$ for 
$n\to\infty$. On the positive side, sufficient conditions for convergence
have been stated in \cite{AKS20}.

In general the recommended way of regularizing by projection is Method II, 
namely the dual least-squares method proposed by Natterer~\cite{Natt77}.
In contrast to Method~I it starts with a sequence of nested subspaces 
$\Y_n\subset\Y$. These can be $\Y_n=\hull\{y_1,\dots,y_n\}$ 
(as suggested in \cite{AKS20}),
but it may be preferable to think of some other subspaces with suitable
approximation properties. Whatever the choice, Method II determines the 
solution $\xdls$ of minimal norm of the problem
\be{dls}
   \text{minimize} \ \norm{P_{\Y_n}\!Kx-w_n}_\Y\,, \qquad w_n=P_{\Y_n}\yd\,,
\ee
where $P_{\Y_n}$ is the orthogonal projection of $\Y$ onto $\Y_n$.
As pointed out in \cite{AKS20}, however, in our setting
$P_{\Y_n}\!K$ is unknown on $\X\setminus\X_n$, and so is its null space, 
and therefore the computation of the 
solution of \req{dls} \emph{of minimal norm} is not possible 
from the given data.

\section{A Bayesian framework}
\label{Sec:Bayes}
In practice the given sequences $(x_k)_k$ and $(y_k)_k$ will often correspond
to a given set of training data, where the $x_k$ are ``plausible images'' 
and $y_k$ are the associated data. In a Bayesian
framework one would argue that $(x_k)_k$ consists of samples from
a given probability distribution on $\X$. In the simplest case this
distribution is Gaussian, i.e.,
\be{Gaussian}
   x \,\sim\, {\cal N}(\mean,\Gamma)
\ee
with mean $\mean\in\X$ and covariance operator $\Gamma\in{\cal L}(\X)$,
the latter being a self-adjoint positive semidefinite compact operator of 
trace class. In this framework it is recommended to replace the 
space $\X$ by the affine subspace $\mean+\E$, where $\E$ is the Cameron-Martin space $\E=\rge{\Gamma^{1/2}}\subset\X$, equipped with the inner product
\bdm
   \scalp{u,u'}_\E \,=\, \scalp{u,\Gamma^{-1}u'}_\X\,, \qquad u,u'\in\E\,,
\edm
cf.~Stuart~\cite[Section~6.3]{Stua10}, and to approximate $\xdagger\in\mean+\E$, 
by minimizing the Tikhonov functional
\be{MAP}
   \text{minimize} \ \norm{Kx-\yd}^2_\Y \,+\, \alpha \norm{x-\mean}^2_\E\,, 
\ee
over $x\in\mean+\E$,
where the regularization parameter $\alpha$ depends on the noise level 
in the data according to \req{noisy} and \req{eta}. Note that if $K$ is a Hilbert-Schmidt operator then
the minimizer of \req{MAP} is the 
\emph{maximum a posteriori estimation} (MAP) of $\xdagger$ under a Gaussian
white noise model for $\eta$ and the assumed distribution of the images;
see \cite{Stua10} or Kaipio and Somersalo~\cite{KaSo05} 
for further details. 

Without any further information on the probability distribution in $\X$
it is natual to use the sample mean and sample covariances
\be{cov}
   \mean_n \,:=\, \frac{1}{n}\,\sum_{k=1}^n x_k \qquad \text{and} \qquad
   \Gamma_n \,:=\, \frac{1}{n} \,\sum_{k=1}^n (x_k-\mean_n)(x_k-\mean_n)^*
\ee
as approximations of $\mean$ and $\Gamma$,
respectively, where for $x\in\X$ we let $x^*$ from the dual space of $\X$ 
be given by $x^*v \,=\, \scalp{x,v}_\X$ for $v\in\X$. 
In other words, one makes the assumption that
\bdm
   x \,\sim\, {\cal N}(\mean_n,\Gamma_n)\,,
\edm 
and defines the Cameron-Martin space $\E_n$ accordingly.
Note that if $p=\dim\X_n$, then $\Gamma_n$ is a positive semidefinite operator 
with
\be{rankGamma}
   p-1 \,\leq\, p' \,:=\, \operatorname{rank}\Gamma_n \,\leq\,\min\{p,n-1\}\,,
\ee
cf.\ the appendix, and $\mean_n+\E_n$ is a $p'$-dimensional affine 
subspace of $\X_n$: either this is $\X_n$ itself, or it is an affine subspace
of $\X_n$ of codimension one -- depending on whether $p'=p$ or $p-1$, 
respectively; compare Figure~\ref{Fig1}. In particular, intuitively speaking, 
every $x\in\X\setminus(\mean+\E_n)$ has ``infinite'' $\E_n$-norm. 

\begin{figure}
\label{Fig1}
\centerline{\includegraphics{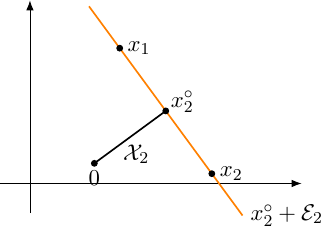}}
\caption{A sketch of the sets $\X_n$ and $\E_n$ for $p=n=2$ and $p'=1$.}
\end{figure}

Accordingly, in this Bayesian framework the dual least-squares method takes 
the following form: Let 
\bdm
   \ymean_n \,=\, \frac{1}{n}\sum_{k=1}^n y_k \,,
\edm
which is close to $K\mean_n$ by virtue of \req{noisy},
and determine the element $u\in\E_n$ of minimal $\E_n$-norm, which solves
\be{Bayes}
   \text{minimize} \ \norm{P_{\Y_n}K_nu - w_n}_\Y\,, \qquad
   w_n \,=\, P_{\Y_n}(\yd-\ymean_n)\,.
\ee
Then return $\mean_n+u$ as approximate solution of the inverse problem.
Note that \req{Bayes} is similar to \req{dls}; it only differs in the definition
of $K_n$ and the use of centered data. This approximation still depends on the choice of $\Y_n$;
we refer to it as Method~III, if
\bdm
   \Y_n \,=\, \hull\{y_1-\ymean_n,\dots,y_n-\ymean_n\}\,,
\edm
and denote the corresponding approximation as $\xbayes$.

The numerical computation of $\xbayes$ is similar to
the computation of $\xlsp$ in Section~\ref{Sec:Setting}, and for this
reason we adopt the corresponding notation,
although the matrices are different from those in Section~\ref{Sec:Setting}. 
\begin{enumerate}
\item
Let the columns of $\Mean_n = [\mean_n,\dots,\mean_n] \in\X_n^n$ 
consist of $n$ identical copies of the sample mean, 
and compute the singular value decomposition
\bdm
   X_n-\Mean_n \,=\, U_n\Lambda \Xi^* \quad \text{with} \quad U_n \in \E_n^{p'},\; \Lambda \in \R^{p' \times p'},\; \Xi^* \in \R^{p' \times n}
\edm
of $X_n-\Mean_n$. 
Then the columns of $U_n=[u_1,\dots,u_{p'}]$ contain the
eigenvectors of the covariance operator $\Gamma_n$ defined in \req{cov} associated with its positive eigenvalues $\lambda_k$, 
the square roots of which are sitting on the diagonal of 
$\Lambda\in\R^{p'\times p'}$ multiplied by $\sqrt{n}$.
This implies that if $\mean_n=0$ then $\sigma_k=\sqrt{n\lambda_k}$ is the singular value
which appears in \req{sigmak}.

Note that the columns of $\Xi$ form an orthonormal basis of the orthogonal
complement of the null space of $X_n-\Mean_n$, and hence, they are all
orthogonal to the vector of all ones, because
\bdm  
   \sum_{k=1}^{n} (x_k-x_n^\circ) \,=\, \sum_{k=1}^{n} x_k \,-\, n x_n^\circ
    \,=\, 0\,.
\edm

With this singular value decomposition it follows that
\bdm
   \norm{u}_{\E_n}^2 \,=\, \scalp{u,\Gamma_n^{-1}u}_\X \,=\,
   \begin{cases}
      n\,\norm{\Lambda^{-1}U_n^*u}_2^2\,, & 
          u\in\E_n=\hull\{u_1,\dots,u_{p'}\}\,, \\[1ex]
      +\infty\,, & \text{else}\,,
   \end{cases}
\edm
i.e.,
\be{J}
   J \,:\, \E_n \,\to\, \R^{p'}\,, \qquad 
   u \,\mapsto\,\xx \,=\, \sqrt{n}\,\Lambda^{-1}U_n^*u\,,
\ee
is an isometry. 
\item
Similar to Section~\ref{Sec:Setting} -- compare with \req{Zn} and \req{Kn} -- we consider now centered data $y_i-\ymean_n$ as input, let
\bdm
   Z_n \,=\, [y_1-\ymean_n,\dots,y_n-\ymean_n]\, \Xi\Lambda^{-1} \,,
\edm
and define
\bdm
   K_n \,=\, Z_nU_n^* \,\approx\, KP_{\E_n}\,.
\edm
Note that here $K_n$ is approximating the restriction of $K$ to the Cameron-Martin space, while in \req{Kn} it has been the restriction of $K$ to $\X_n$.

One advantage of this framework is that here we can avoid computing $K_n$ or $Z_n$ 
and rather evaluate
\bdm
   L_n \,=\, [y_1-\ymean_n,\dots,y_n-\ymean_n]\, \Xi
\edm
instead. Take note that by virtue of \req{J}, and because $U_n$ is orthogonal,
\be{LnKn}
   L_n \,=\, Z_n\Lambda \,=\, K_n U_n\Lambda \,=\, \sqrt{n}\,K_n J^{-1}\,.
\ee
\item
Accordingly, in the noisy case the least squares problem~\req{Bayes} becomes
\bdm
   \text{minimize} \ \norm{K_nu - w_n}_\Y \,=\, \norm{K_nJ^{-1}\xx - w_n}_\Y\,, 
   \qquad
   w_n \,=\, P_{\Y_n}(\yd-\ymean_n)\,,
\edm
over $u\in\E_n$, respectively $\xx\in\R^{p'}$, in such a way that $\xx=Ju$ 
has minimal Euclidean norm. 
The solution of this minimization problem is
given by $\xx = L_n^\dagger(\sqrt{n}\,w_n)$, and the corresponding approximation 
of $\xdagger$ is
\be{xnIII}
   \xbayes \,=\, \mean_n \,+\, U_n\Lambda L_n^\dagger(\yd - \ymean_n)\,.
\ee
The computation of $\xbayes$ requires the singular value decomposition 
of $L_n$. 

If we want to compute the MAP estimate~\req{MAP} of $\xdagger$ 
instead, then we have to replace $L_n^\dagger$ by $(L_n^*L_n+\alpha I)^{-1}L_n^*$
in \req{xnIII}. Again, the evaluation of this estimate simplifies when the
singular value decomposition of $L_n$ is computed first.
\end{enumerate} 

It is a nice feature of Method~III that it circumvents the inversion of the
singular values in $\Lambda$ which was necessary in the evaluation
of $K_n$ in \req{Kn} for computing $\xlsp$. Thus the computation of $\xbayes$
is more stable in the presence of clustered images.

\section{A basic error estimate}
\label{Sec:error}
\begin{figure}
\label{Fig2}
\centerline{\includegraphics{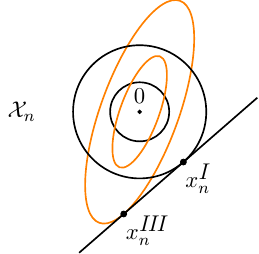}}
\caption{Different minimum norm solutions according to
the norm in $\X$ (Method I) and the Cameron-Martin norm (Method III), when $\mean_n=0$ and $\ymean_n=0$. The solid line shows the intersection of the solution manifold of $K_nx=P_{\Y_n}y$ with $\X_n$; the circles are the level lines of the norm in $\X$ and
the ellipses correspond to the Cameron-Martin norm.}
\end{figure}

Note that when $\mean_n=0$ and $\ymean_n=0$ then 
$\E_n=\X_n$, but their norms are different,
and Method III is the least-squares
projection method using the Cameron-Martin norm.
If $K_n|_{\X_n}$ is injective, then Method I and
Method III coincide, cf.~Lemma~\ref{Lem:MoorePenrose} 
below: they both determine the unique
solution of $Kx=P_{\Y_n}y$ in $\X_n$. 
In particular, this implies that the aforementioned Seidman example 
is also relevant for Method~III.
If $K_n|_{\X_n}$ fails to be injective then the
methods are different, because they return the 
respective solutions of minimal norm; 
see Figure~\ref{Fig2}.

Finally, the reason that the connection to the dual least squares method is somewhat
weaker, is due to the fact that $\mean_n+\E_n$ varies 
with $n$, and does not belong to the orthogonal complement of the null space of $K$,
in general.

\begin{lemma}
\label{Lem:MoorePenrose}
Let $U_n$, $\Lambda$, $K_n$ and $L_n$ be defined as in 
Section~\ref{Sec:Bayes}, and let $K_n|_{\E_n}$ be injective.
Then $K_n^\dagger=U_n\Lambda L_n^\dagger$.
\end{lemma}

\begin{proof}
We show that $A=U_n\Lambda L_n^\dagger$ satisfies all four defining equations 
for the Moore-Penrose generalized inverse, listed, e.g., 
in \cite[Proposition~2.3]{EHN96}. 
We start by noting that it follows from \req{LnKn} that
\be{KnA}
   K_nA \,=\, K_nU_n\Lambda L_n^\dagger \,=\, L_nL_n^\dagger 
        \,=\, P_{\rge{L_n}} \,=\, P_{\rge{K_n}}\,,
\ee
because the range spaces $\rge{L_n}$ and $\rge{K_n}$ of $L_n$ and $K_n$ 
coincide. Accordingly, we have 
\bdm
   AK_nA \,=\, U_n\Lambda L_n^\dagger P_{\rge{L_n}}
   \,=\, U_n\Lambda L_n^\dagger \,=\, A\,,
\edm
and
\bdm
   K_nAK_n \,=\, P_{\rge{K_n}}K_n \,=\, K_n\,.
\edm

Note that these three identities hold independent of whether $K_n|_{\E_n}$
is injective, or not. We only need this assumption to prove the remaining
identity. In fact, due to this assumption the range of $L_n$ has dimension
$p'$, and hence, $L_n^\dagger L_n$ is the identity in $\R^{p'}$. 
It therefore follows from \req{LnKn} and \req{J} that
\bdm
   AK_n|_{\E_n} \,=\, \frac{1}{\sqrt{n}}\, U_n\Lambda L_n^\dagger L_n J
   \,=\, \frac{1}{\sqrt{n}}\, U_n\Lambda J \,=\, J^{-1}J 
\edm
is the identity in $\E_n$, whereas $AK_nx\,=\, 0$ for every $x\in\E_n^\bot$.
This establishes that
\bdm
   AK_n\,=\,P_{\E_n}\,.
\edm

Accordingly, $A$ is the Moore-Penrose generalized inverse of $K_n$.
\end{proof}


The assumption that $K_n|_{\E_n}$ is injective is not a very restrictive one.
For example, this is likely to be the case when $K:\X\to\Y$ is known to be 
injective. Also, if the range of $[y_1-\ymean_n,\dots,y_n-\ymean_n]$ has
the largest possible dimension $n-1$, so that its null space is spanned by 
$[1,\dots,1]^T$, then the dimension of $\rge{L_n}=\rge{K_n}$ is $p'$,
and again, $K_n|_{\E_n}$ is injective.
This second case occurs with high probability if the noise components in 
the data $y_k$ are uncorrelated random variables.

\begin{proposition}
\label{Prop:Thm3.24}
Let $\yd=K\xdagger$ for some $\xdagger\in\X$, and assume that $K_n|_{\E_n}$
is injective. Furthermore, let $\mu_n$ be the smallest singular value of $K_n$.
Then
\be{xbayes-error}
   \xbayes-x^\dagger 
   \,=\, (I-P_{\E_n})(\mean_n-x^\dagger) \,+\, d_n\,,
\ee
where 
\bdm
   \norm{d_n}_\X \,\leq\, \frac{1}{\mu_n}\,
   \Bigl(\norm{(K-K_n)(\mean_n-\xdagger)}_\Y 
         \,+\, \norm{\ymean_n - K\mean_n)}_\Y\Bigr)\,.
\edm 
\end{proposition}

\begin{proof}
According to \req{xnIII} and Lemma~\ref{Lem:MoorePenrose} there holds
\bdm
   \xbayes - \mean_n \,=\, K_n^\dagger K_n(\xdagger-\mean_n)
   \,+\, K_n^\dagger r_n\,,
\edm
where
\be{rn}
   r_n \,=\, \yd - \ymean_n - K_n(\xdagger - \mean_n)
       \,=\, (K-K_n)(\xdagger-\mean_n) \,+\, K\mean_n - \ymean_n\,.
\ee
Since $K_n^\dagger K_n=P_{\E_n}$ because of the injectivity of $K_n|_{\E_n}$,
we thus obtain
\bdm
   \xbayes \,=\, P_{\E_n}\xdagger \,+\, (I-P_{\E_n})\mean_n \,+\, d_n
\edm
with 
\be{dn}
   d_n \,=\, K_n^\dagger r_n\,.
\ee
From this we readily conclude the error representation~\req{xbayes-error},
where the estimate of $\norm{d_n}_\X$ follows from \req{rn} and the fact, that
$\norm{K_n^\dagger}_{\Y\to\X}=1/\mu_n$.
\end{proof}

The error representation~\req{xbayes-error} with $d_n$ of \req{dn} extends
the corresponding formula for the least-squares projection as stated, e.g.,
by Kaltenbacher and Offtermatt~\cite[(2.9)]{KaOf12}.



\appendix
\section{On the rank of the sample covariance matrix}
\label{Sec:Appendix}
Here we discuss the sharpness of the inequalities in \req{rankGamma}
concerning the rank $p'$ of the sample covariance operator $\Gamma_n$
and the dimension $p$ of $\X_n$.

Let $u\in\X$ be any vector. Then it follows that
\bdm 
   \scalp{u,\Gamma_n u}_\X \,=\, 
   \frac{1}{n}\sum_{k=1}^n \bigl|\scalp{x_k-\mean_n,u}_\X\bigr|^2 \,\geq\, 0\,,
\edm
with equality if and only if $u\bot\hull\{x_k-\mean_n\}$. This means
that 
\bdm
   p' \,=\, \dim \bigl(\hull \{x_k-\mean_n\}\bigr)\,.
\edm
Since this span is a subset of $\X_n$ we obviously have $p'\leq p$. 
Further, the set $\{x_k-\mean_n\}$ is always linearly dependent, 
because these vectors add up to zero; hence $p'\leq n-1$. 
On the other hand, since
\bdm
   \hull\bigl(\{x_k-\mean_n\} \,\cup\,\{\mean_n\}\bigr)
\edm
is the entire subspace $\X_n$, we necessarily have $p'\geq p-1$.

In particular, if $p=n$, then $p'=n-1$. 

\begin{example}
\label{Ex1}
\rm
Let
\bdm
   x_1 \,=\, \begin{cmatrix} 1 \\ 0 \\ 0 \end{cmatrix}\,, \quad
   x_2 \,=\, \begin{cmatrix} 0 \\ 1 \\ 0 \end{cmatrix}\,, \quad 
   \text{and} \quad x_3 = 5x_2-4x_1\,.
\edm
In this example we have $p=2$ and $\mean_3=2x_2-x_1$. 
It follows that $x_1-\mean_3 \,=\, 2(x_1-x_2)$,
$x_2-\mean_3 \,=\, x_1-x_2$, and $x_3-\mean_3 \,=\, 3(x_2-x_1)$ are all
collinear, hence $p'=1=p-1$.
This example is similar to the one which is displayed in Figure~\ref{Fig1}, except that here all the three points $x_k$ are lying in the same one-dimensional affine subspace
\fin
\end{example}

Note that we always have $p'=p$, when the sample mean vanishes. But this is
not the only possibility:

\begin{example}
\label{Ex2}
\rm
Let
\bdm
   x_1 \,=\, \begin{cmatrix} 1 \\ 0 \\ 0 \end{cmatrix}\,, \quad
   x_2 \,=\, \begin{cmatrix} 0 \\ 1 \\ 0 \end{cmatrix}\,, \quad 
   \text{and} \quad x_3 = 2(x_2+x_1)\,.
\edm
In this example we have $p=2$ again. This time $\mean_3=x_2+x_1$, so that 
$x_1-\mean_3 \,=\, -x_2$ and $x_2-\mean_3 = -x_1$
are linearly independent. Accordingly, in this example we have $p'=p=2$.
\fin
\end{example}

\subsection*{Acknowledgements}
%
%
This research was funded in whole, or in part, by the Austrian Science Fund
(FWF) 10.55776/P34981 -- New Inverse Problems of Super-Resolved Microscopy (NIPSUM).
For the purpose of open access, the author has applied a CC BY public copyright
license to any Author Accepted Manuscript version arising from this submission.
This research was funded in whole or in part by the Austrian Science Fund (FWF)
SFB 10.55776/F68 ``Tomography Across the Scales'', project F6807-N36
(Tomography with Uncertainties). For open access purposes, the author has
applied a CC BY public copyright license to any author-accepted manuscript
version arising from this submission.
The financial support by the Austrian Federal Ministry for Digital and Economic
Affairs, the National Foundation for Research, Technology and Development and the Christian Doppler
Research Association is gratefully acknowledged.


\end{document}